\documentclass[11pt]{article}
\usepackage[margin=1in]{geometry}
\usepackage{graphicx} 
\usepackage{amsmath, amsfonts,amssymb}
\usepackage{amsthm}
\usepackage{amstext}
\usepackage{amsopn}
\usepackage{appendix}
\usepackage{caption}    
\usepackage{subcaption} 
\usepackage{tikz}
\usepackage{booktabs,multirow}

\usepackage{multirow}
\usepackage{array}
\usepackage{hyphenat}  
\usepackage[shortlabels]{enumitem}

\usepackage{graphicx}
\usepackage{bm}
\newtheorem{theorem}{Theorem}[section]
\newtheorem{lemma}[theorem]{Lemma}
\newtheorem{proposition}[theorem]{Proposition}
\newtheorem{corollary}[theorem]{Corollary}
\newtheorem{remark}[theorem]{Remark}

\def\MM{\mathcal{M}}
\def\YY{\mathcal{Y}}

\theoremstyle{definition}
\newtheorem{definition}[theorem]{Definition}
\newtheorem{example}[theorem]{Example}

\begin{document}

\numberwithin{equation}{section}
\title{{\vspace{-80pt}} Optimal Harvesting  in Stream Networks: \\ Maximizing Biomass and Yield
}
\date{}
\author{~}

\maketitle

\vspace{-0.2cm}

\begin{center}{\large
Tung D. Nguyen\\[1mm]
Department of Mathematics, University of California Los Angeles\\
Los Angeles, CA 90024, USA\\[3mm]

Zhisheng Shuai\\[1mm]
Department of Mathematics, University of Central Florida\\
Orlando, FL 32816, USA\\[3mm]

Tingting Tang\\[1mm]
Department of Mathematics and Statistics, San Diego State University\\
San Diego, CA 92182, USA\\[3mm]

Amy Veprauskas\\[1mm]
Department of Mathematics, University of Louisiana at Lafayette\\
Lafayette, LA 70501, USA\\[3mm]

Yixiang Wu\\[1mm]
Department of Mathematical Sciences, Middle Tennessee State University\\
Murfreesboro, TN 37132, USA\\[3mm]

Ying Zhou\\[1mm]
Department of Mathematical Sciences, Lafayette College\\
Easton, PA 18042, USA
}
\end{center}

\smallskip

\begin{abstract}
    In this study, we develop a metapopulation model framework to identify optimal harvesting strategies for a population in a stream network. We consider two distinct optimization objectives: maximization of total biomass and maximization of total yield, under the constraint of a fixed total harvesting effort. We examine in detail the special case of a two-patch network and fully characterize the optimal strategies for each objective. We show that when the population growth rate exceeds a critical threshold, a single harvesting strategy can simultaneously maximize both objectives. For general $n$-patch networks with homogeneous growth rates across patches, we focus on the regime of large growth rates and demonstrate that the optimal harvesting strategy selects patches according to their intraspecific competition rates and an \textit{effective net flow} metric determined by network connectivity parameters. \\
    
\end{abstract}

\newpage
\section{Introduction}

Harvesting strategies are an important part of managing renewable resources such as fish stocks. While harvesting provides essential benefits including food and economic livelihood, poorly managed harvesting can cause overexploitation and resource depletion \cite{hilborn2011overfishing, agriculture2018state}. It is therefore critical to identify sustainable harvesting strategies that strike a balance between harvesting and species preservation. Mathematical models offer a useful framework for addressing this challenge. Concepts such as maximum sustainable yield in the classic bioeconomic models \cite{clark2005mathematical, gordon1954, schaefer1954} have played a foundational role in learning how to optimize yield without depleting the resource. The characterization and analysis of optimal harvesting strategies have received considerable attention in the mathematical and ecological study of renewable resource management, including investigations in homogeneous environments \cite{brauer1975constant,brauer1976constant,brauer1979harvesting, dunkel1970maximum, dunkel2006maximum,swan1975some, wan1983optimal}, periodic environments \cite{dong2007optimal,fan1998optimal, feng2023periodic, liu2022optimal, xiao2006optimal, xu2005harvesting}, models with impulsive harvesting \cite{angelova2000optimization,berezansky2004impulsive,braverman2008continuous,cid2014harvest,shuai2007optimization,tang2006optimal,zhang2003optimal,zhang2006optimal}, and spatially heterogeneous settings formulated through partial differential equations \cite{bai2005gilpin, braverman2009optimal, cui2017effect, ding2009optimal,herrera2010spatial, kelly2016optimal, kelly2019marine, korobenko2009logistic, korobenko2013persistence, neubert2003marine, oruganti2002diffusive} or metapopulation systems of ordinary differential equations \cite{auger2022increase,elbetch2023nonlinear,gao2022total,gonzalez2011two, nguyen2024connectivity, pilyugin2016effectiveness, reurik2025connectivity, supriatna1999harvesting, takashina2015maximum, tuck2000marine, wu2020dispersal,zhang2015effects}.

While many harvesting models assume well-mixed populations, it is beneficial to consider the effect of space and dispersal in determining optimal harvesting strategies \cite{ding2009optimal, gonzalez2011two, herrera2010spatial, kelly2016optimal, supriatna1999harvesting, takashina2015maximum, tuck2000marine}. In \cite{herrera2010spatial}, the authors showed that ignoring spatial dynamics can lead to suboptimal results. In \cite{takashina2015maximum}, the authors demonstrated that spatially explicit harvest models can yield  lower maximum sustainable yields than their non-spatial counterparts, raising concerns of overharvesting when spatial heterogeneity is neglected. Spatially explicit models are especially important when designing marine projected areas \cite{gaines2010designing, gonzalez2011two, kelly2016optimal, moeller2015economically, neubert2003marine, moeller2015economically} or considering the effect of spatial heterogeneity \cite{kelly2016optimal} on populations. A spatially explicit model also allows us to investigate how the spatial distribution of harvesting efforts, rather than the total sum of harvesting effort, affects the population or the harvesting yield.

Besides optimizing the total yield, increasing or preserving total biomass of a fish population can also be a desirable outcome in fisheries management \cite{jensen2002maximum} because the decline or collapse of a stock is always a concern for fisheries \cite{lauck1998implementing}. In some ways, harvesting and preserving may not be conflictive goals: in \cite{sanchirico2001bioeconomic}, the authors showed that both the sustainable yield and the sustainable aggregate biomass of the fish population can be increased by the creation of no-harvest marine reserves. Can we find such win-win scenarios in the context of optimization? And if we cannot optimize both the yield and the biomass, how should we direct harvesting efforts when prioritizing one over the other?

In this paper, we develop a metapopulation model for a population that resides in a stream network consisting of connected habitat patches. 
Our goals are to find optimal ways to distribute harvesting efforts in the patches so as to maximize either the total biomass or the total yield across the patches.  For a homogeneous two-patch model, we find that when resources are abundant, maximum total biomass is achieved by harvesting exclusively on the downstream patch; otherwise, harvesting exclusively in the upstream patch maximizes biomass (Theorem 3.5; also see Theorem~3.4). Harvesting exclusively in the downstream patch also maximizes yield when growth rates are high and advection is sufficiently strong relative to the total harvest (Theorem 3.11). When these parameter conditions are not met, however, optimal strategies for maximizing the yield can vary (Figure 2). For $n$-patch networks with heterogeneous intraspecific competition rates, the two optimization objectives lead to opposite strategies: maximizing biomass favors concentrating harvesting on patches with larger competition rates (Corollary 4.3), while maximizing yield favors concentrating on patches with smaller competition rates (Corollary 4.10). Within these patches, we find that the optimal distribution of harvesting efforts for both the biomass and yield objective depends on an effective net flow metric (Definition 4.4).

The rest of the paper is organized as follows. In Section \ref{section model}, the model formulation for a stream network with harvesting is presented and and the biomass and yield optimization problems are defined. Section \ref{section constrained harvest} considers a special case where the network consists of only two patches. For this scenario we examine optimal harvesting strategies for both heterogeneous and homogeneous patches. Section \ref{sec n patch} extends the results to an $n$-patch network. Finally, we discuss the findings and directions of future work in Section \ref{section discussion}.

\section{Model formulation and optimization problems}\label{section model}
The classic logistic model takes the form  \begin{equation}\label{eq-logistic}
    u'=ru\Big(1-\frac{u}{K}\Big),
\end{equation}
where $u=u(t)$ denotes the population size (or density) at time $t$, $r$ is the intrinsic growth rate, and $K>0$ is the environmental carrying capacity. This formulation presumes a nonnegative intrinsic growth rate ($r\ge 0$) and thus becomes invalid when $r<0$ as it may produce negative population size. 

To allow for both positive and negative intrinsic growth, we adopt the following logistic-type modification:
\begin{equation}
    u'=u(r-cu),
\label{eq:model}\end{equation}
where $r\in \mathbb{R}$ may be either positive or negative, and $c>0$ measures the strength of intraspecies competition. Model \eqref{eq:model} may admit two equilibria: the trivial equilibrium $u=0$, and, when $r>0$, the positive equilibrium $u^*=\frac{r}{c}$. The stability properties are straightforward: if $r\le 0$, the trivial equilibrium $u=0$ is globally asymptotically stable in $[0, +\infty)$; if $r> 0$, then the trivial equilibrium becomes unstable and the positive equilibrium $u^*$ is globally asymptotically stable in $(0,+\infty)$. Biologically, when $r>0$, the ratio $r/c$ can be interpreted as the environmental carrying capacity, analogous to $K$ in model \eqref{eq-logistic}.

Building on the modified logistic model \eqref{eq:model}, we introduce the following metapopulation framework with explicit population control in a landscape of $n$ patches:
\begin{equation}\label{n patch system}
u_i' = u_i (r_i - c_i u_i) - h_i u_i + \sum_{j=1}^n (a_{ij}u_j - a_{ji} u_i),  \quad i=1,2,\ldots, n,
\end{equation}
where $u_i=u_i(t)$ denotes the population density in patch $i$ at time $t$, $r_i$ is the intrinsic growth rate, $c_i>0$ reflects the intensity of intraspecies competition, $h_i$ represents the harvesting or stocking effort applied to patch $i$, and $a_{ij} \ge 0$ denotes the dispersal rate from patch $j$ to patch $i$. Assume that the $n\times n$ movement matrix $A=[a_{ij}]$ is irreducible so that an individual can move directly or indirectly from any two ordered patches. As with the base model \eqref{eq:model}, the dynamics of the metapopulation model \eqref{n patch system} are characterized by the stability of the trivial equilibrium $E_0=(0,0,\ldots, 0)$ and the positive equilibrium $E^*=(u_1^*, u_2^*, \ldots, u_n^*)$: if $s(J) < 0$, then the trivial equilibrium $E_0$ is globally asymptotically stable; if $s(J)>0$, then the positive equilibrium $E^*$ is globally asymptotically stable.

Despite considerable progress in recent decades, major challenges persist in {characterizing optimal harvesting strategies, }largely due to the intrinsic complexity of the optimization problems and the heterogeneity inherent of the modeling framework. We illustrate these difficulties here for the case of unconstrained harvesting. When $n=1$, the maximum sustainable yield (MSY) is $$\mathcal{Y}=\frac{r^2}{4c} \; ,$$ 
which is attained at the harvesting rate $h=\frac{r}{2}$, yielding the equilibrium population $u^*=\frac{r}{2c}$. If $h>r$, the trivial equilibrium $u=0$ is globally asymptotically stable and the population goes to extinction. 

For general $n$, the positive equilibrium 
$u^*=(u_1^*, \ldots, u_n^*)$ satisfies
\begin{equation}\label{eq-eq}
    u_i^*[(r_i-h_i) - c_iu_i^*] + \sum_j (a_{ij} u_j^* - a_{ji} u_i^*) = 0, \qquad i=1,\ldots,n.
\end{equation}
Summing over $i$ and using the identity
\begin{equation*}
    \sum_{i,j} (a_{ij} u_j^* - a_{ji} u_i^*) = 0,
\end{equation*}
we obtain
\begin{equation*}
    \sum_i u_i^*[(r_i-h_i)-c_iu_i^*] = 0.
\end{equation*}
Rearranging yields an explicit expression for the total yield:
\begin{equation}
    \mathcal{Y} = \sum_i h_i u_i^* = \sum_i r_i u_i^* - \sum_i c_i (u_i^*)^2. 
\end{equation}
The right-hand side is maximized when $u_i^* = \frac{r_i}{2c_i}, i=1,\ldots,n$.  Substituting these values into the equilibrium equations \eqref{eq-eq} shows that the harvesting rates achieving maximal yield satisfy
\begin{equation}\label{eq-hs}
    h_i = \dfrac{r_i}{2} + \sum_j \Big(a_{ij}\frac{r_j/c_j}{r_i/c_i} - a_{ji}\Big), \qquad i=1,\ldots,n.
\end{equation}
These optimal harvesting strategies, which may be positive (harvesting) or negative (stocking), reflect a combination of local patch properties ($r_i$ and $c_i$) and the overall network connectivity among patches, encoded in the dispersal rates $a_{ij}$. The resulting  metapopulation-level maximum sustainable yield (MSY) is given by
\begin{equation}
    \mathcal{Y} = \sum_i \frac{r_i^2}{4c_i} \; .
\end{equation}
 In the special case when $r_i=r$ for all $i$, the harvesting effort distribution depends solely on the sum $\sum_j \Big(a_{ij}\frac{c_i}{c_j} - a_{ji}\Big)$, which can be regarded as a ``weighted" net flow into a node $i$. We show later in Section \ref{sec n patch} that a similar quantity (Definition \ref{def:netflow}) also plays a vital role in maximizing the total biomass and the yield in the constrained harvesting problem.

Under the additional constrain $h_i\ge 0$ (i.e., stocking is not allowed), the MSY problem becomes considerably more complex.  In the remaining of the paper, we aim to provide a systematic classification of which habitat patches to control, and to what extent, in order to achieve optimal management objectives in this case. In particular, we consider the following two optimization problems

\begin{enumerate}
\item {\bf Maximizing total biomass:} 
$$\max_{h_i \ge 0} \; \liminf_{t\to\infty} \sum_{i=1}^n u_i(t),$$

\item {\bf Maximizing sustainable harvesting yield:}  
$$ \max_{h_i\geq 0} \liminf_{t\to\infty} \sum_{i=1}^n h_i u_i(t),$$
\end{enumerate}
subject to the harvesting constraint $\sum_{i=1}^n h_i = H$.

 To focus the analysis, we consider two cases, when  heterogeneity arises solely from the network of connections between otherwise homogeneous patches and when the patches are also heterogeneous. These cases allows us to isolate and rigorously examine the role of connectivity and patch heterogeneity in shaping both intrinsic dynamics and the effectiveness of control strategies. Our results highlight not only the critical influence of structural heterogeneity on population outcomes, but also its implications for the design of optimal interventions in fragmented or interconnected ecological systems.

\section{Constrained harvesting in a two-patch system with biased movement}\label{section constrained harvest}

In this section we consider a constrained harvesting problem in a two-patch system with biased movement. We make the following assumptions:
\begin{enumerate}
    
    \item [(H1)] We call patch 1 the \textit{upstream patch} and patch 2 the \textit{downstream patch} and assume that the movement rate from the upstream patch to the downstream patch is $d+q$, and the movement rate from the downstream patch to the upstream patch is $d$ where $d, q>0$.
    \item [(H2)] The total harvesting effort is fixed and there is no stocking, i.e. $h_1+h_2=H>0$ and $h_1,h_2\geq 0$.
\end{enumerate}
To specify the distribution of harvesting effort, we introduce a new parameter $\theta\in[0,1]$ such that:
\begin{equation}\label{eqn:theta}
h_1=\theta H \quad \text{and} \quad h_2=(1-\theta)H.
\end{equation}
Combining assumptions (H1), (H2) and equation \eqref{eqn:theta}, the ODE system we consider is
\begin{align}\label{eqn:ODEs_2patch}
    &u_1'=u_1(r_1-\theta H -c_1u_1) -(d+q)u_1+du_2\nonumber\\
    &u_2'=u_2(r_2-(1-\theta)H -c_2u_2)+(d+q)u_1-du_2.
\end{align}
First, we present a sufficient condition for the species to persist in both patches.
\begin{lemma}\label{lem:persist}
    Suppose the intrinsic growth rates $r_1$ and $r_2$ satisfy
    \[
    \frac{d}{2d+q}r_1+\frac{d+q}{2d+q}r_2>H\frac{d+q}{2d+q} =: r_{\text{crit}}.
    \]
    Then there exists a unique  positive equilibrium $\bm u^*$ which is asymptotically stable.
\end{lemma}
\begin{proof}
    From \cite{chen2022two}, we can obtain a lower bound for the spectral bound of the Jacobian matrix for the linearization of  system \eqref{eqn:ODEs_2patch} at the trivial equilibrium
    \[
    \rho \geq (r_1-\theta H)\phi_1 + (r_2-(1-\theta)H)\phi_2,
    \]
    where $(\phi_1,\phi_2)$ is the normalized eigenvector of $L$ corresponding to the eigenvalue $0$. It is easy to check that 
    \[
    (\phi_1,\phi_2)=\bigg(\frac{d}{2d+q},\frac{d+q}{2d+q}\bigg).
    \]
    Thus we have 
    \begin{align*}
    \rho &\geq (r_1-\theta H)\frac{d}{2d+q} + (r_2-(1-\theta)H)\frac{d+q}{2d+q} \\
    &= \frac{d}{2d+q}r_1+\frac{d+q}{2d+q}r_2 - H\bigg(\theta\frac{d}{2d+q}+(1-\theta)\frac{d+q}{2d+q}\bigg)\\
    &\geq \frac{d}{2d+q}r_1+\frac{d+q}{2d+q}r_2-r_{\text{crit}}
    \end{align*}
    since $\theta\in[0,1]$. Therefore a sufficient condition for persistence is that 
     \[
    \frac{d}{2d+q}r_1+\frac{d+q}{2d+q}r_2>H\frac{d+q}{2d+q} =: r_{\text{crit}}.
    \]
\end{proof}

Assuming that the species persists, let $\bm u^*(\theta) = (u_1^*(\theta),u_2^*(\theta))$ be the unique positive equilibrium of system \eqref{eqn:ODEs_2patch}. For notational convenience, we omit the function arguments and the star and just write $(u_1,u_2)$ for the positive equilibrium. Thus $(u_1,u_2)$ is the solution of the following system
\begin{subequations}
\begin{align}
    &u_1(r_1-\theta H -c_1u_1) -(d+q)u_1+du_2=0\label{eqn:u1}\\
    &u_2(r_2-(1-\theta)H -c_2u_2)+(d+q)u_1-du_2=0\label{eqn:u2}.
\end{align}    
\end{subequations}

\subsection{Maximizing the total biomass when the patches are heterogeneous}\label{sec:biomass_het}
We first consider the optimal harvesting strategy if the goal is to maximize the (remaining) total biomass $\MM=u_1+u_2$. We start with a lemma that provides an expression for $\MM'(\theta)$.

\begin{lemma}\label{lem:M'_hetero}
We have the equality
\begin{equation}\label{eq:M'_het}
    \MM'(\theta)= -\frac{H(u_1+u_2)((d+q)u_1^2-du_2^2)+Hu_1^2u_2^2(c_2-c_1)}{\det(A)},
\end{equation}
where 
\begin{equation}\label{matrix A_het}
A=\begin{bmatrix}
    -(c_1u_1^2+du_2) & du_1\\
    (d+q)u_2 & -(c_2u_2^2+(d+q)u_1)
\end{bmatrix}.
\end{equation}
\end{lemma}
\begin{proof}
Differentiating equations \eqref{eqn:u1} and \eqref{eqn:u2} in terms of $\theta$ yields another pair of equations
\begin{subequations}
    \begin{align}
           &u_1'(r_1-\theta H-c_1u_1)-u_1(H+c_1u_1')-(d+q)u_1'+du_2'=0\label{eqn:u1'_het}\\
           &u_2'(r_2-(1-\theta)H-c_2u_2) - u_2(-H+c_2u_2')+(d+q)u_1'-du_2'=0.\label{eqn:u2'_het}
    \end{align}
\end{subequations}
By taking $u_1\times \eqref{eqn:u1'_het}-u_1'\times \eqref{eqn:u1}$ and $u_2\times \eqref{eqn:u2'_het}-u_2'\times \eqref{eqn:u2}$ we obtain
\begin{subequations}
    \begin{align*}
        &-u_1^2(H+c_1u_1') +du_2'u_1-du_1'u_2=0\\
        &-u_2^2(-H+c_2u_2')+(d+q)u_1'u_2-(d+q)u_2'u_1=0.
    \end{align*}
\end{subequations}
Rearranging terms results in
\begin{align*}
    A\begin{pmatrix} u_1' \\u_2'\end{pmatrix} = \begin{pmatrix} Hu_1^2\\-Hu_2^2\end{pmatrix},
\end{align*}
where the matrix $A$ is given by equation \eqref{matrix A_het}.
It is easy to check that $\det(A)>0$. Solving this system yields
\begin{align*}
    \begin{pmatrix} u_1' \\u_2'\end{pmatrix} &=\frac{-1}{\det(A)} \begin{bmatrix}
    c_2u_2^2+(d+q)u_1 & du_1\\
    (d+q)u_2 & c_1u_1^2+du_2
\end{bmatrix}\begin{pmatrix} Hu_1^2\\-Hu_2^2\end{pmatrix}\nonumber\\
&=\frac{-1}{\det(A)}\begin{pmatrix} Hu_1^2(c_2u_2^2+(d+q)u_1) - Hdu_2^2u_1\\H(d+q)u_1^2u_2-Hu_2^2(c_1u_1^2+du_2)\end{pmatrix}.\label{eqn:u1'u2'_explicit}
\end{align*}
Adding $u_1'$ and $u_2'$ and simplifying gives us equation \eqref{eq:M'_het}.
\end{proof}

From Lemma \ref{lem:M'_hetero}, there are two factors affecting the sign of $\MM'$ (and thus the strategy to maximize $\MM$).
    \begin{itemize}
        \item The first factor is the difference $(d+q)u_1^2-du_2^2$. We will show later that this difference depends on the parameters from both individual patches and the movement network.
        \item The second factor is the difference $c_2-c_1$, which only depends on the intraspecific competition rate of individual patches.
    \end{itemize}

\begin{remark}\normalfont
While determining the sign of $\MM'$ is challenging in general, there are cases when one factor dominates the other, and $\MM'(\theta)$ is guaranteed to have one sign for any $\theta\in[0,1]$.

Firstly, when $u_1$ and $u_2$ are large and of the same order of magnitude, the sign of $\MM'(\theta)$ is determined by $c_2-c_1$. Thus the strategy to maximize $\MM$ depends  on the individual patch parameters and not on the movement rates. In particular, if $c_2>c_1$, then $\MM$ is maximized when $\theta=0$ and, if $c_1>c_2$, then $\MM$ is maximized when $\theta=1$. In other words, to maximize the total biomass, we  concentrate the harvesting effort on the patch with the higher intraspecific competition rate. This observation is generalized to $n$-patch systems in Section \ref{sec n patch}.

On the other hand, if $c_1=c_2=c$, then the sign of $\MM'(\theta)$ is determined by the sign of $(d+q)u_1^2-du_2^2$. We show in the next theorem that this sign depends on a ``weighted difference" between the intrinsic growth rates $r_1$ and $r_2$.
\end{remark}

\begin{theorem}\label{theorem:M'}
 Consider system \eqref{eqn:ODEs_2patch} and assume the persistence condition in Lemma \ref{lem:persist}. Assume further that $c_1=c_2=c$.
    \begin{enumerate}
        \item[(a)] Suppose the growth rates $r_1$ and $r_2$ satisfy 
        \[\frac{r_1\sqrt{d+q}-r_2\sqrt{d}}{\sqrt{d+q}-\sqrt{d}}> 2d+q +\frac{\sqrt{d+q}}{\sqrt{d+q}-\sqrt{d}}H=:r_M.\]
        Then $\MM'(\theta)< 0$ for any $\theta\in[0,1]$. Thus $\MM(\theta)$ is maximized at $\theta=0$, i.e. when the harvesting effort is concentrated on the downstream patch.
        \item[(b)]  Suppose the growth rates $r_1$ and $r_2$ satisfy 
        \[\frac{r_1\sqrt{d+q}-r_2\sqrt{d}}{\sqrt{d+q}-\sqrt{d}}< 2d+q -\frac{\sqrt{d}}{\sqrt{d+q}-\sqrt{d}}H=:r_m.\]
        Then $\MM'(\theta)> 0$ for any $\theta\in[0,1]$. Thus $\MM(\theta)$ is maximized at $\theta=1$, i.e. when the harvesting effort is concentrated on the upstream patch.
        \item[(c)] Suppose 
        \[
        r_m\leq \frac{r_1\sqrt{d+q}-r_2\sqrt{d}}{\sqrt{d+q}-\sqrt{d}}\leq r_M
        \]
        then there exists a unique $\theta^*\in[0,1]$ such that $\MM'(\theta^*)=0$. Furthermore, $\theta^*$ is a minimum of $\MM(\theta)$ in the interval $[0,1]$ and $\MM(\theta)$ is maximized at either $\theta=0$ or $\theta=1$, i.e. when the harvesting effort is concentrated  on either patch.
        \end{enumerate}
\end{theorem}

\begin{proof}
    Let $t:=\frac{u_1}{u_2}$. Dividing equations \eqref{eqn:u1} and \eqref{eqn:u2} by $u_1$ and $u_2$, respectively, yields
 \begin{subequations}
     \begin{align}
         &r_1-\theta H -(d+q) + \frac{d}{t}=cu_1 \label{eqn:div_u1_het}\\
         &r_2-(1-\theta)H -d + (d+q)t = cu_2. \label{eqn:div_u2_het}
     \end{align}
\end{subequations}
By taking \eqref{eqn:div_u1_het}$\div$\eqref{eqn:div_u2_het}, we obtain
\[
\frac{r_1-\theta H -(d+q) + \frac{d}{t}}{r_2-(1-\theta)H -d + (d+q)t}=t.
\]
Thus $t$ satisfies
\[
f(t,\theta) : = t(r_2-(1-\theta)H -d + (d+q)t) - (r_1-\theta H -(d+q) + \frac{d}{t}) =0.
\]
Consider a fixed value of $\theta$. Note that $r_2-(1-\theta)H -d + (d+q)t>0$ due to \eqref{eqn:div_u2_het}, so $f(t,\theta)$ is an increasing function with respect to $t$.

From Lemma \ref{lem:M'_hetero} and the assumption that $c_1=c_2=c$, the sign of $\MM'(\theta)$ is the opposite of the sign of $(d+q)u_1^2-du_2^2=(d+q)u_2^2(t^2-\frac{d}{d+q})$. Thus we wish to examine the condition for $t>\sqrt{\frac{d}{d+q}}$ and $t<\sqrt{\frac{d}{d+q}}$. To this end, we use direct calculations and evaluate
\begin{equation}
    f\bigg(\sqrt{\frac{d}{d+q}},\theta\bigg) = \frac{\sqrt{d+q}-\sqrt{d}}{\sqrt{d+q}}\bigg(g(\theta)-\frac{r_1\sqrt{d+q}-r_2\sqrt{d}}{\sqrt{d+q}-\sqrt{d}}\bigg),
\end{equation}
where $g(\theta):=2d+q+\frac{\theta\sqrt{d+q}-(1-\theta)\sqrt{d}}{\sqrt{d+q}-\sqrt{d}}H$. It is easy to see that $g(\theta)$ is increasing in $\theta$. Furthermore, $g(0)=r_m$ and $g(1)=r_M$, where $r_m$ and $r_M$ are defined in Theorem \ref{theorem:M'}. We now consider three cases. 

\vspace{.1in}
\noindent\textbf{Case (a):} Suppose that
$\frac{r_1\sqrt{d+q}-r_2\sqrt{d}}{\sqrt{d+q}-\sqrt{d}}>r_M$.

Since $g(\theta)$ is increasing in $\theta$, we have $\frac{r_1\sqrt{d+q}-r_2\sqrt{d}}{\sqrt{d+q}-\sqrt{d}}>r_M=g(1)\geq g(\theta)$. As a result, 
$f\big(\sqrt{\frac{d}{d+q}},\theta\big) <0$, which implies $t>\sqrt{\frac{d}{d+q}}$. From Lemma \ref{lem:M'_hetero} and the assumption that $c_1=c_2=c$, we  conclude that $\MM'(\theta)<0$ and $\MM(\theta)$ is maximized at $\theta=0$.

\vspace{.1in}
\noindent\textbf{Case (b):} Suppose that
$\frac{r_1\sqrt{d+q}-r_2\sqrt{d}}{\sqrt{d+q}-\sqrt{d}}<r_m$.

Since $g(\theta)$ is increasing in $\theta$, we have $\frac{r_1\sqrt{d+q}-r_2\sqrt{d}}{\sqrt{d+q}-\sqrt{d}}<r_m=g(0)\leq g(\theta)$. As a result, 
$f\big(\sqrt{\frac{d}{d+q}},\theta\big) >0$, which implies $t<\sqrt{\frac{d}{d+q}}$. From Lemma \ref{lem:M'_hetero} and the assumption that $c_1=c_2=c$, we  conclude that $\MM'(\theta)>0$ and $\MM(\theta)$ is maximized at $\theta=1$.

\vspace{.1in}
\noindent\textbf{Case (c):} Suppose that $r_m\leq \frac{r_1\sqrt{d+q}-r_2\sqrt{d}}{\sqrt{d+q}-\sqrt{d}}\leq r_M$.
\vspace{.1in}

In this case, there exists a unique $\theta^*\in[0,1]$ such that $\frac{r_1\sqrt{d+q}-r_2\sqrt{d}}{\sqrt{d+q}-\sqrt{d}}=g(\theta^*)$. Thus $f\big(\sqrt{\frac{d}{d+q}},\theta^*\big) =0$, which  implies $t=\sqrt{\frac{d}{d+q}}$ when $\theta=\theta^*$. From Lemma \ref{lem:M'_hetero} and the assumption that $c_1=c_2=c$ we must have $\MM'(\theta^*)=0$. 

Furthermore, for any $\theta\in[0,\theta^*]$, we have $\frac{r_1\sqrt{d+q}-r_2\sqrt{d}}{\sqrt{d+q}-\sqrt{d}}>g(\theta)$ and thus $\MM'(\theta)<0$. Using an identical argument, we also have $\MM'(\theta)>0$ for any $\theta\in[\theta^*,1]$. Thus in the interval $[0,1]$, the total biomass $\MM(\theta)$ has only one local extremum $\theta^*$, which is a minimum. This also implies that $\MM(\theta)$ is maximized at either $\theta=0$ or $\theta=1$.
\end{proof}

\subsection{Maximizing the total biomass when the patches are homogeneous}\label{sec:biomass_hom}

In this subsection, we consider the case when the two patches are homogeneous and we show that it is possible to establish a biomass-maximizing strategy for any choice of parameters. More precisely, in addition to assumptions (H1) and (H2), we make the following assumption:

\begin{enumerate}
\item [(H3)] The intrinsic growth rates and intraspecies competition rates of the two patches are the same, i.e. $r_1=r_2=r$ and $c_1=c_2=c$.
\end{enumerate}
Under assumption (H3), system \eqref{eqn:ODEs_2patch} becomes
\begin{align}\label{eqn:ODEs_2patch_hom}
    &u_1'=u_1(r-\theta H -cu_1) -(d+q)u_1+du_2\nonumber\\
    &u_2'=u_2(r-(1-\theta)H -cu_2)+(d+q)u_1-du_2.
\end{align}
The main result for this system is stated in the following theorem.
\begin{theorem}\label{thm:max_M}
    Consider system \eqref{eqn:ODEs_2patch_hom}.
    \begin{enumerate}
        \item[(a)] If $r>2d+q+\frac{H}{2}$,  then $\MM(\theta)\leq \MM(0)$ for any $\theta\in[0,1]$. In other words, the biomass is maximized when harvesting effort is concentrated on the downstream patch.
        \item[(b)] If $r_{\text{crit}}\leq r<2d+q+\frac{H}{2}$, then $\mathcal{M}(\theta)\leq \mathcal{M}(1)$ for any $\theta\in[0,1]$.  In other words, the biomass is maximized when harvesting effort is concentrated on the upstream patch.
        \item[(c)] If $r=2d+q+\frac{H}{2}$, then $\MM(\theta)\leq \MM(0)=\MM(1)$ for any $\theta\in[0,1]$.  In other words, the biomass is maximized when harvesting effort is concentrated exclusively on either the upstream or the downstream patch.
    \end{enumerate}
\end{theorem}

To prove Theorem \ref{thm:max_M}, we need a series of technical lemmas. We start with a corollary of Theorem \ref{theorem:M'}. Under the assumption (H3), the weighted difference in Theorem \ref{theorem:M'} becomes
\[
\frac{r_1\sqrt{d+q}-r_2\sqrt{d}}{\sqrt{d+q}-\sqrt{d}} = r.
\]
Thus, we have the following corollary.

\begin{corollary}\label{cor:M'}
      Consider system \eqref{eqn:ODEs_2patch_hom}.
    \begin{enumerate}
        \item[(a)] (Large growth rate) Suppose the growth rate $r$ satisfies $r>r_M$.
        Then $\MM'(\theta)< 0$ for any $\theta\in[0,1]$. 
        \item[(b)] (Small growth rate) Suppose the growth rate $r$ satisfies $r<r_m$
        Then $\MM'(\theta)> 0$ for any $\theta\in[0,1]$.
        \item[(c)] (Intermediate growth rate) Suppose $r_m\leq r\leq r_M$, then there exists a unique $\theta^*\in[0,1]$ such that $\MM'(\theta^*)=0$. Furthermore, $\theta^*$ is a minimum of $\MM(\theta)$ in the interval $[0,1]$ and $\MM(\theta)$ is maximized at either $\theta=0$ or $\theta=1$.
        \end{enumerate} 
\end{corollary}

Focusing on the case when $r_m\leq r \leq r_M$, our next lemma establishes when the total biomass has the same value at $\theta=0$ and $\theta=1$. 

\begin{lemma}\label{lem:M0_M1}
    If $\MM(0)=\MM(1)$, we must have $r=2d+q+\frac{H}{2}$.
\end{lemma}
\begin{proof}
    Let $(a_1,a_2)$ be the positive equilibrium when $\theta=0$, i.e. it is the solution of 
    \begin{subequations}
    \begin{align}
        &a_1(r-c a_1)-(d+q)a_1+da_2=0\label{eqn:a1}\\
        &a_2(r-H-ca_2)+(d+q)a_1-da_2=0.\label{eqn:a2}
    \end{align}   
    \end{subequations}
    Adding equations \eqref{eqn:a1} and \eqref{eqn:a2} yields
    \begin{equation}\label{eqn:a1+a2}
     r(a_1+a_2)-c(a_1^2+a_2^2)=a_2H.   
    \end{equation}
    Let $(b_1,b_2)$ be the positive equilibrium when $\theta=1$, i.e. it is the solution of 
    \begin{subequations}
    \begin{align}
        &b_1(r-H-c b_1)-(d+q)b_1+db_2=0\label{eqn:b1}\\
        &b_2(r-cb_2)+(d+q)b_1-db_2=0.\label{eqn:b2}
    \end{align}   
    \end{subequations}
    Adding equations \eqref{eqn:b1} and \eqref{eqn:b2} yields
    \begin{equation}\label{eqn:b1+b2}
     r(b_1+b_2)-c(b_1^2+b_2^2)=b_1H.   
    \end{equation}
    Since $\MM(0)=\MM(1)$, we have $a_1+a_2=b_1+b_2$. Subtracting \eqref{eqn:a1+a2} from \eqref{eqn:b1+b2} yields
    \begin{equation}\label{eqn:a2-b1}
    c(b_1^2+b_2^2-a_1^2-a_2^2) = H(a_2-b_1).
    \end{equation}
    By using $a_2-b_1=b_2-a_1$ and $a_1-b_1=b_2-a_2$ in equation \eqref{eqn:a2-b1}, we obtain
    \begin{equation}\label{eqn:a2-b1_simplified}
    c(a_2-b_1)(b_2+a_1-b_1-a_2)=H(a_2-b_1) \implies 2c(a_2-b_1)(a_1-b_1) = H(a_2-b_1).
    \end{equation}
    We first argue that $a_2\neq b_1$. Assume by contradiction that $a_2=b_1$. This further implies $a_1=b_2$. Subtracting \eqref{eqn:a1} from \eqref{eqn:b2} yields
    \[
    (d+q)a_2-da_1 + (d+q)a_1 -da_2=0,
    \]
    which leads to $a_1+a_2=0$, a contradiction. Thus equation \eqref{eqn:a2-b1_simplified} must imply $2c(a_1-b_1)=H$. Subtracting \eqref{eqn:b1} from \eqref{eqn:a1} and using $a_2-b_2=b_1-a_1$ yields
    \begin{equation}\label{eqn:a1-b1}
    r(a_1-b_1)-c(a_1-b_1)(a_1+b_1) -(2d+q)(a_1-b_1)+Hb_1=0.
    \end{equation}
    By substituting $c(a_1-b_1)=H/2$ into equation \eqref{eqn:a1-b1}, we obtain
    \[
    (a_1-b_1)\bigg(r-2d-q-\frac{H}{2}\bigg)=0.
    \]
    It is easy to see that $a_1\neq b_1$ from equation \eqref{eqn:a1-b1}, thus we must have $r=2d+q+\frac{H}{2}$.
\end{proof}

Next, we show that the converse of Lemma \ref{lem:M0_M1} is also true.

\begin{lemma}\label{lem:M0_M1_converse}
If $r=2d+q+\frac{H}{2}$, then we have $\MM(0)=\MM(1)$.    
\end{lemma}
\begin{proof}
Similar to Lemma \ref{lem:M0_M1}, let $(a_1,a_2)$ be the positive equilibrium when $\theta=0$, i.e. it is the solution of
\begin{subequations}
    \begin{align}
        &a_1(2d+q+\frac{H}{2}-c a_1)-(d+q)a_1+da_2=0\label{eqn:a1'}\\
        &a_2(2d+q-\frac{H}{2}-ca_2)+(d+q)a_1-da_2=0.\label{eqn:a2'}
    \end{align}   
    \end{subequations}
Subtracting \eqref{eqn:a2'} from \eqref{eqn:a1'} yields
\[
\frac{H}{2}(a_1+a_2)-c(a_1+a_2)(a_1-a_2) - q(a_1+a_2)=0 \implies a_2-a_1=\frac{q-\frac{H}{2}}{c}.
\]
We substitute this into \eqref{eqn:a1'} and solve for $a_1$ to obtain
\[
a_1=\frac{2d+\frac{H}{2}+\sqrt{(2d+\frac{H}{2})^2+4d(q-\frac{H}{2})}}{2c},
\]
which gives us
\[
\MM(0)=a_1+a_2=2a_1+\frac{q-\frac{H}{2}}{c}=\frac{2d+q+\sqrt{(2d+\frac{H}{2})^2+4d(q-\frac{H}{2})}}{2c}.
\]
Using similar calculations for the case $\theta=1$ yields 
\[
\MM(1)=\frac{2d+q+\sqrt{(2d-\frac{H}{2})^2+4d(q+\frac{H}{2})}}{2c}.
\]
Finally, it is easy to check that $\MM(0)=\MM(1)$.
\end{proof}

\begin{lemma}\label{lem:medium_growth}
When $2d+q+\frac{H}{2}<r\leq r_M$, the total biomass $\MM(\theta)$ is maximized at $\theta=0$. When $r_m\leq r<2d+q+\frac{H}{2}$, the total biomass $\MM(\theta)$ is maximized at $\theta=1$. 
\end{lemma}
\begin{proof}
    From Lemma \ref{lem:M0_M1}, we have $\{(d,q,r,c,H):\MM(0)=\MM(1)\} = \{(d,q,r,c,H): r=2d+q+\frac{H}{2}\}$. Since $\MM(0)-\MM(1)$ depends continuously on the parameters $d,q,r,c,H$, it must have the same sign in each region $R_1=\{(d,q,r,c,H): r<2d+q+\frac{H}{2}\}$ and $R_2=\{(d,q,r,c,H): r>2d+q+\frac{H}{2}\}$. 

    It is easy to see that $\{(d,q,r,c,H): r<r_m\}\subset R_1$, so Corollary \ref{cor:M'} (b) implies $\MM(0)<\MM(1)$ in $R_1$. Similarly, since $\{(d,q,r,c,H): r>r_M\}\subset R_2$, Corollary \ref{cor:M'} (a) implies $\MM(0)>\MM(1)$ in $R_2$.  Finally, using Corollary \ref{cor:M'} (c), we can conclude that $\MM(\theta)$ is maximized at $\theta=0$ when $2d+q+\frac{H}{2}<r\leq r_M$ and maximized at $\theta=1$ when $r_m\leq r<2d+q+\frac{H}{2}$.
\end{proof}

\begin{figure}[htbp]
    \centering
\begin{subfigure}[t]{0.48\textwidth}
        \centering
\includegraphics[width=1\linewidth]{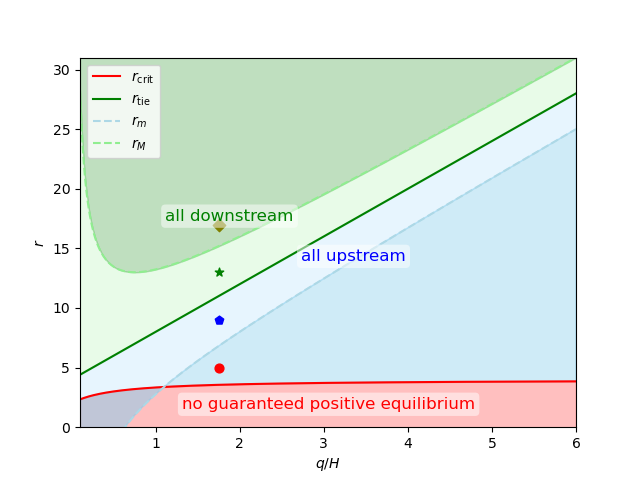}
\vspace{.09in}
        \caption{Illustration of an instance of Theorem~\ref{thm:max_M}.}
        \label{fig:left_figure_B}
    \end{subfigure}
\begin{subfigure}[t]{0.48\textwidth}
\vspace{-5.5cm}
    \includegraphics[width=0.45\linewidth]{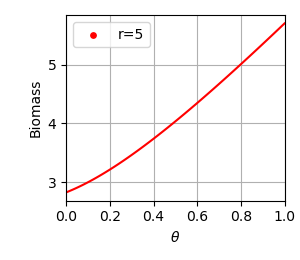}
    \includegraphics[width=0.45\linewidth]{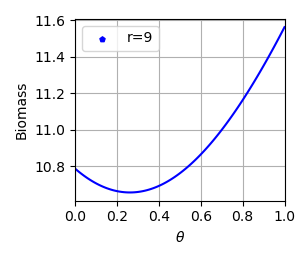}
        
        \vspace{-0.5em}

    \includegraphics[width=0.45\linewidth]{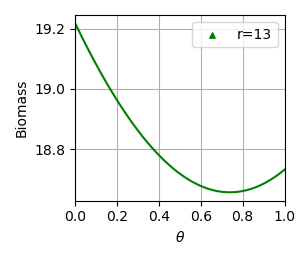}
    \includegraphics[width=0.45\linewidth]{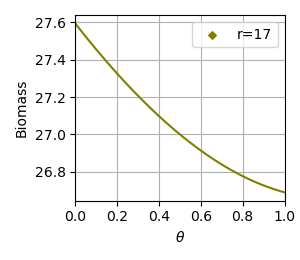}
    \caption{Biomass vs harvesting}        \label{fig:right_figure_B}

        \end{subfigure}
    \caption{In these two figures, $r_{tie}=2d+q+\frac{H}{2}$, $d = 1$ and $H=4$, the four critical curves $r_{crit}$, $r_m,r_M,r_{tie}$ in Figure \ref{fig:left_figure_B} separate the region to  areas where the strategies for maximizing biomass are different. Four specific points are chosen with $q/H=1.75$ and different $r$ values. Figure \ref{fig:right_figure_B} shows the total biomass against the harvest parameter $\theta$ at these four points with $c=1$ and $q=7$. }
    \label{fig:2patch_biomass}
\end{figure}

The proof of Theorem \ref{thm:max_M} now follows directly from Corollary  \ref{cor:M'}, Lemma \ref{lem:M0_M1_converse}, and Lemma \ref{lem:medium_growth}. \textbf{Figure 1} provides an illustration of Theorem~\ref{thm:max_M}. In panel~(a), the green curve \(r_{\mathrm{tie}}=2d+q+\tfrac{H}{2}\) partitions the \((q/H,\; r)\)-plane into management regimes. Above \(r_{\mathrm{tie}}\) (green shading), the biomass-maximizing strategy is to harvest exclusively downstream (\(\theta=0\)). Biologically, advection feeds individuals into the downstream patch, where densities, thus crowding losses are highest, so removing biomass there frees capacity while preserving upstream production. Below \(r_{\mathrm{tie}}\) but above the persistence threshold \(r_{\mathrm{crit}}=\tfrac{H(d+q)}{2d+q}\) (blue shading), harvesting exclusively upstream (\(\theta=1\)) is optimal. Here, the intrinsic growth is relatively modest, and protecting the upstream ``source'' patch sustains system-wide production before individuals are flushed downstream. The red band below \(r_{\mathrm{crit}}\) indicates parameters with no guaranteed positive equilibrium. The four markers at \(q/H=1.75\) (red, blue, green, gold) correspond to panel~(b), which plots total biomass versus the harvest split \(\theta\) (with \(d=1\), \(H=4\), \(c=1\), \(q=7\)). Consistent with Corollary~\ref{cor:M'}, for small \(r\) (e.g., \(r=5\), red), biomass increases as effort shifts upstream---protecting the source boosts overall standing stock. For large \(r\) (e.g., \(r=17\), gold), biomass decreases as effort moves upstream---downstream harvest removes the surplus where crowding is strongest. At intermediate \(r\) (e.g., \(r=13\), green), the biomass first declines and then rises with \(\theta\), reflecting a transition in residence time and crowding balance across patches as the budget is reallocated. Boundaries \(r_m\) and \(r_M\) (dashed lines in panel (a)) demarcate these monotonicity regimes predicted by Corollary~\ref{cor:M'}.

\subsection{Maximizing the yield when the patches are homogeneous}\label{sec:yield}

We now turn to the optimal harvesting strategy if the goal is to maximize the yield $\mathcal{Y}(\theta)=\theta Hu_1+(1-\theta)Hu_2$. Here, we focus on the case in which the patches are homogeneous so that we are able to obtain analytical results.

Under assumption (H3), we derive an expression for $\YY'(\theta)$ in the following lemma.

\begin{lemma}\label{lem:Y'}
We have the equality
\[
\frac{1}{H}\YY'(\theta)=\bigg[\frac{H(1-2\theta)-q}{c}+ \frac{H(1-2\theta)cu_1^2u_2^2}{\det(A)}\bigg]- ((d+q)u_1^2-du_2^2)\bigg[\frac{1}{cu_1u_2}+\frac{H(\theta u_1+(1-\theta)u_2)}{\det(A)}\bigg],
\]
where 
\begin{equation}\label{matrix A_hom}
A=\begin{bmatrix}
    -(cu_1^2+du_2) & du_1\\
    (d+q)u_2 & -(cu_2^2+(d+q)u_1)
\end{bmatrix}.
\end{equation}

\end{lemma}
\begin{proof}
Differentiating the equation $\YY(\theta)=\theta Hu_1+(1-\theta)Hu_2$ yields
\[
\frac{1}{H}\YY'(\theta)= u_1-u_2 + \theta u_1' +(1-\theta)u_2'.
\]
By subtracting \eqref{eqn:div_u2_het} from \eqref{eqn:div_u1_het} and using assumption (H3), we obtain
\begin{equation}\label{eqn:u1-u2}
    c(u_1-u_2)=H(1-2\theta)-q +d\frac{u_2}{u_1} - (d+q)u_1 \implies u_1-u_2=\frac{(1-2\theta)H-q}{c} - \frac{(d+q)u_1^2-du_2^2}{cu_1u_2}.
\end{equation}
Using the expressions for $u_1'$ and $u_2'$ in the proof of Lemma \ref{lem:M'_hetero} yields
\begin{equation}\label{eqn:theta_u1_u2}
    \theta u_1' +(1-\theta)u_2'= - \frac{H}{\det(A)}[(2\theta-1)cu_1^2u_2^2+ (\theta u_1+(1-\theta)u_2)((d+q)u_1^2-du_2^2)].
\end{equation}
Combining equations \eqref{eqn:u1-u2} and \eqref{eqn:theta_u1_u2} and rearranging terms gives us the desired result.
\end{proof}

Lemma \ref{lem:Y'} allows us to establish a sufficient condition in which $\YY'(\theta)<0$ and thus the yield is maximized when harvesting exclusively downstream.

\begin{theorem}\label{thm:max_Y_downstream}
Suppose that $r>r_M$, where $r_M$ is defined in Lemma \ref{theorem:M'}, and that $q\geq 2H$. Then we have $\YY(\theta)\leq \YY(0)$ for any $\theta\in[0,1]$. In other words, the yield is maximized when harvesting is applied exclusively to the downstream patch.    
\end{theorem}
\begin{proof}
    Since $r>r_M$, we have $t=\frac{u_1}{u_2}>\sqrt{\frac{d}{d+q}}$ from the proof of Theorem \ref{theorem:M'} (a). Thus we have 
    \[
    ((d+q)u_1^2-du_2^2)\bigg[\frac{1}{cu_1u_2}+\frac{H(\theta u_1+(1-\theta)u_2)}{\det(A)}\bigg]>0
    \]
    since $\det(A)>0$. 
    If $\theta \geq 1/2$, then it is clear that 
    \[
    \frac{H(1-2\theta)-q}{c}+ \frac{H(1-2\theta)cu_1^2u_2^2}{\det(A)} <0.
    \]
    On the other hand, if $\theta <1/2$, we first notice
    \[\det(A)=(cu_1^2+du_2)(cu_2^2+(d+q)u_1)-d(d+q)u_1u_2>c^2u_1^2u_2^2.
    \]
    Thus we have
    \[
    \frac{H(1-2\theta)-q}{c}+ \frac{H(1-2\theta)cu_1^2u_2^2}{\det(A)} < \frac{2H(1-2\theta)-q}{c} \leq \frac{2H-q}{c}\leq 0.
    \]
    Therefore, by using Lemma \ref{lem:Y'}, we can conclude $\YY'(\theta) <0$ for any $\theta\in[0,1]$. This further implies $\YY(\theta)$ is maximized when $\theta=0$.
\end{proof}

The result in Theorem \ref{thm:max_Y_downstream} is quite intuitive. The condition $q\geq 2H$ (large advection compared to the harvesting effort) allows the biomass of the downstream patch to be large compared to the upstream patch even after harvesting. The condition $r>r_M$ (large growth rate) allows the total biomass to be maximized when harvesting exclusively downstream. Together, they naturally suggests the yield is also maximized when harvesting downstream.  

When the condition in Theorem \ref{thm:max_Y_downstream} is not satisfied, we can still use Lemma \ref{lem:Y'} to obtain some partial information on the distribution of harvesting effort that maximizes the yield. We state two such results in the corollary below. The proof of the corollary follows directly from Lemma \ref{lem:Y'}.

\begin{corollary}\label{cor:max_Y}
We have the following statements on the yield-maximizing effort distribution.
    \begin{enumerate}
        \item[(a)] (Large growth rate) Suppose that $r>r_M$. If $\theta\geq 1/2$, then $\YY'(\theta)<0$. Thus the yield is maximized at some $\theta^*<1/2$. In other words, the yield is maximized when we harvest more on the downstream patch.
        \item[(b)] (Small growth rate) Suppose that $r<r_m$. If $H(1-2\theta)\geq q$, or equivalently $\theta\leq\frac{1}{2}-\frac{q}{2H}$, then $\YY'(\theta)>0$. Thus the yield is maximized at some $\theta^*>\frac{1}{2}-\frac{q}{2H}$. 
    \end{enumerate}

\end{corollary}

As hinted in Corollary \ref{cor:max_Y}, the yield can follow many different patterns when the condition in Theorem \ref{thm:max_Y_downstream} is not satisfied. \textbf{Figure 2} presents numerical simulations illustrating these. In Figure \ref{fig:left_figure_Y}, the curves $r_m$ (blue) and $r_M$ (orange) partition the $(q/H,r)$-plane and delimit where the yield-maximizing split $\theta^\ast$ (fraction of budget upstream) may occur. In the green strip, the sufficient condition $q\ge 2H$ together with large $r$ ensures that the maximum yield is achieved by harvesting only downstream ($\theta^\ast=0$). Below $r_m$ and $q<H$, the optimal harvesting strategy cannot be to harvest all downstream from case $(b)$ in Corollary~\ref{cor:max_Y}. Below $r_m$ and $q>H$ (lower tan region), the condition $\theta>\frac{1}{2}-\frac{q}{2H}$ in Corollary~\ref{cor:max_Y} is reduced to $\theta>0$, thus the corollary provides no information regarding best harvesting strategy. Figure~\ref{fig:right_Y} confirms these behaviors at $q/H=0.75$ (with $d=1$, $H=4$, $c=1$, and $q=3$) for four choices of $r$: for $r=1$, the yield peaks when harvesting all upstream $\theta= 1$; for $r=3$, the maximizer is interior around $\theta\approx 0.66$; for $r=11$ and $r=15$, the maximizers occur near $\theta= 0.06$ an $\theta = 0.08$, respectively. This also indicates that the optimum harvesting strategy for yield may change non-monotonously as $r$ increases, further demonstrating the complexity of optimizing yield in a stream.

\begin{figure}
\begin{subfigure}[t]{0.48\textwidth}
        \centering
\includegraphics[width=1\linewidth]{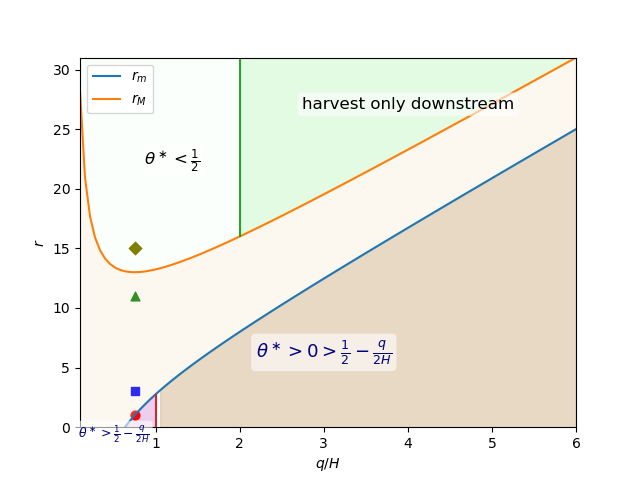}
        \caption{Illustration of an instance of Theorem~\ref{thm:max_Y_downstream} and Corollary~\ref{cor:max_Y}}
        \label{fig:left_figure_Y}
    \end{subfigure}
\begin{subfigure}[t]{0.48\textwidth}
\vspace{-5.5cm}
    \includegraphics[width=0.48\linewidth]{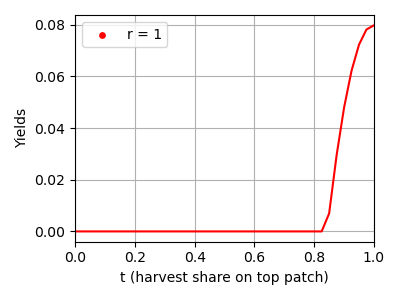}
    \includegraphics[width=0.48\linewidth]{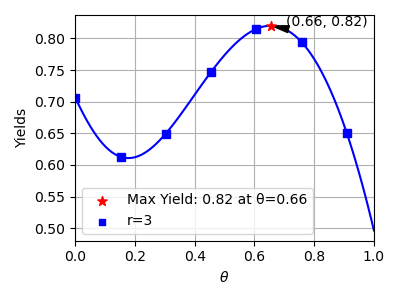}
        
        \vspace{-0.5em}

    \includegraphics[width=0.48\linewidth]{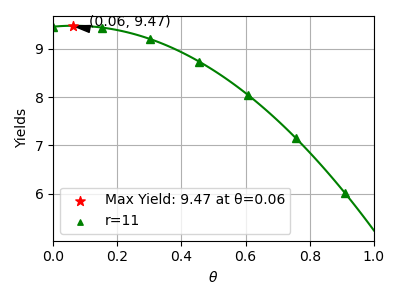}
    \includegraphics[width=0.48\linewidth]{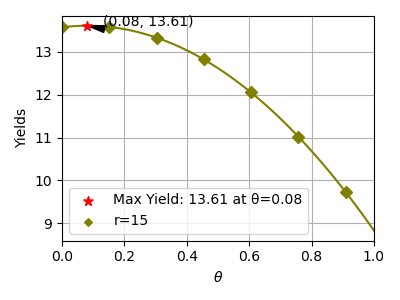}
    \caption{Yield vs harvesting}\label{fig:right_Y}
        \end{subfigure}
    \caption{In these two figures, $d = 1$ and $H=4$,  four specific points are chosen with $q/H=0.75$ and different $r$ values. Figure \ref{fig:right_Y} shows the yield against the harvest parameter $\theta$ at these four points with $c=1$ and $q=3$.}
    \label{fig:2patch_Y}
\end{figure}

\section{Constrained harvesting in $n-$patch systems}\label{sec n patch}

In this section, we generalize some of the results in Section \ref{section constrained harvest} to the $n$-patch system \eqref{n patch system}. In particular, we focus on the regime where the equilibrium values of $u_i$ are large. For convenience, we assume that the intrinsic growth rates are the same in each patch, i.e. $r_1=\dots=r_n=r$ and focus on the effects of the intraspecies competition rates and the movement network on the total biomass and the yield.

\subsection{Maximizing the total biomass}
\subsubsection{Concentrating on the patches with the highest intraspecies competition rate}

Without loss of generality, that is by renumbering patches as needed, we may assume that $c_1=\dots=c_{\ell}=c>c_i$ for $i=\ell+1,\dots,n$, i.e. the first $\ell$ patches have the strongest intraspecies competition rate.

We first show that to maximize the remaining biomass, it is better to concentrate the harvesting effort in patch $1,\dots,\ell$. To see this, let the harvesting effort in each patch be given as follows:
\begin{align}\label{hi}
&h_i=\begin{cases}
   \beta_i(1-\theta)H \quad &\text{for} \quad i=1,\dots,\ell\\
    \alpha_i\theta H \quad &\text{for} \quad i=\ell+1,\dots,n
\end{cases}
\end{align}
where $\sum_{i=1}^\ell \beta_i=1$, $\sum_{i=\ell+1}^n \alpha_i=1$, and $\theta\in[0,1]$. We now show that for any choice of $\bm\alpha \in \mathbb{R}^{n-\ell}$ and $\bm\beta\in\mathbb{R}^\ell$, we must have $\lim_{r\to\infty}\MM'<0$, where the differentiation is with respect to $\theta$. 

We start with some necessary calculations. For $i=\ell+1,\dots,n$, we have 
\begin{subequations}
\begin{align}
 &u_i(r-\alpha_i\theta H-c_iu_i)  + \sum_j (a_{ij} u_j - a_{ji} u_i)=0. \label{eq:u_alpha}\\
 &u_i'(r-\alpha_i\theta H- c_iu_i)-u_i(\alpha_i H+ c_iu_i') + \sum_j (a_{ij} u_j' - a_{ji} u_i')=0,\label{eq:u'_alpha}
\end{align}    
\end{subequations}
where the differentiation is with respect to $\theta$. For $i=1,\dots,\ell$, we have 
\begin{subequations}
\begin{align}
 &u_i(r-\beta_i(1-\theta) H-c_iu_i)  + \sum_j (a_{ij} u_j - a_{ji} u_i)=0. \label{eq:u_beta}\\
 &u_i'(r-\beta_i(1-\theta) H- c_iu_i)+u_i(\beta_i H- c_iu_i') + \sum_j (a_{ij} u_j' - a_{ji} u_i')=0,\label{eq:u'_beta}
\end{align}    
Taking $u_i\times\eqref{eq:u'_alpha} - u_i'\times\eqref{eq:u_alpha}$ yields
\begin{equation}\label{eq:u_alpha_2}
    -u_i^2(\alpha_i H+c_iu_i')+\sum_ja_{ij}(u_j'u_i-u_ju_i')=0 \quad \text{for} \quad i=\ell+1,\dots,n.
\end{equation}
Similarly, we have
\begin{equation}\label{eq:u_beta_2}
    -u_{i}^2(-\beta_iH+c_{i}u_{i}')+\sum_ja_{ij}(u_j'u_{i}-u_ju_{i}')=0 \quad \text{for} \quad i=1,\dots,\ell.
\end{equation}
\end{subequations}
We apply these calculations in the following auxillary results. 

\begin{lemma}\label{lemma:u_alpha}
We have
    \[
    \lim_{r\to\infty} \frac{u_i}{r} = \frac{1}{c_i}.
    \]
    Furthermore the convergence is uniform in $\bm\alpha\in [0, 1]^{n-\ell}$, $\bm\beta\in[0,1]^{\ell}$ and $\theta\in [0, 1]$.
\end{lemma}

\begin{proof}
    Let $\kappa=1/r$ and $y_i=u_i/r$ for $i=1, \dots, n$. Then, dividing both sides of \eqref{eq:u_alpha} and \eqref{eq:u_beta} by $r^2$, we obtain
    \[\bm f (\kappa, \theta, \bm\alpha, \bm\beta, \bm y) = \bf 0,\]

    where $\bm f=(f_1, \dots, f_n): \mathbb{R}\times\mathbb{R}\times \mathbb{R}^{n-\ell}\times \mathbb{R}^{\ell}\times \mathbb{R}^n\rightarrow \mathbb{R}^n$ is given by 
    \begin{align*}
        f_i(\kappa, \theta, \bm\alpha, \bm\beta, \bm y) & =y_i(1-\kappa\alpha_i\theta H-c_i y_i)+\kappa\sum_j(a_{ij}y_j-a_{ji}y_i) \quad \text{for} \quad  i=\ell+1,\dots, n,\\ f_{i}(\kappa, \theta, \bm\alpha, \bm\beta, \bm y) & =y_{i}(1-\kappa(1-\beta_i)\theta H-c_{i} y_{i})+\kappa\sum_j(a_{ij}y_j-a_{ji}y_{i}) \quad \text{for} \quad  i=1, \dots, \ell. 
    \end{align*}

     Fix $\theta_0\in[0, 1]$, $\bm \beta_0=(\beta_{01},\dots,\beta_{0\ell})\in \mathbb{R}^\ell$, and $\bm\alpha_0=(\alpha_{0(\ell+1)},\dots, \alpha_{0n})\in\mathbb{R}^{n-\ell}$. Let $\bm y_0=(y_{01},\dots, y_{0n})=(1/c_1,\dots, 1/c_n)$.  It is easy to check that 
    $$
\bm f(0, \theta_0, \bm\alpha_0, \bm\beta_0, \bm y_0)=\bm 0. 
    $$
    Then, we have 
    $$
D_{\bm y} \bm f(0, \theta_0, \bm\alpha_0, \bm\beta_0, \bm y_0)=-\bm I. 
    $$
By the implicit function theorem, there exist $\delta>0$, open sets $U\subset \mathbb{R}$, $V\subset \mathbb{R}^{n-\ell}$ and $W\subset \mathbb{R}^\ell$ with $\theta_0\in U$, $\bm\alpha_0\in V$, $\bm\beta_0\in W$, and a unique continuous function $\bm g: [0, \delta]\times U\times V\times W\to \mathbb{R}^n$ such that $\bm y=\bm g(\kappa, \theta, \bm\alpha, \bm\beta)$ satisfies 
$$
\bm f(\kappa, \theta, \bm\alpha, \bm\beta, \bm g(\kappa,\theta,\bm\alpha, \bm\beta))=\bm 0, \ \ (\kappa, \theta, \bm\alpha,\bm\beta)\in [0, \delta]\times U\times V\times W,
$$
and 
$$
\bm y_0=\bm g(0, \theta_0, \bm\alpha_0, \bm\beta_0). 
$$
By the compactness of $[0, 1]$, $[0, 1]^{n-\ell}$ and $[0, 1]^{\ell}$, there exists $\delta_m>0$ such that all the solutions of $\bm f(\kappa, \theta, \bm \alpha, \bm\beta, \bm y)=\bm 0$ for $\kappa\in [0, \delta_m]$, $\theta\in [0, 1]$, $\bm\alpha\in [0, 1]^{n-\ell}$ and $\bm\beta\in [0, 1]^\ell$ satisfy $\bm y=\bm g(\kappa, \theta, \bm\alpha,\bm\beta)$. Moreover, $\bm y_0=\bm g(0, \theta, \bm\alpha,\bm\beta)$ for $\theta\in [0, 1]$, $\bm \alpha\in [0, 1]^{n-\ell}$ and $\bm\beta\in[0, 1]^{\ell}$. Hence, we may choose $\delta_m$ small such that $\bm g$ is positive. 

Since $g$ is continuous on the compact set $[0, \delta_m]\times [0, 1]\times \mathbb{R}^{n-\ell}\times\mathbb{R}^\ell$, it is uniformly continuous.
This implies $\lim_{\kappa\to 0^+} \bm g(\kappa, \theta, \bm\alpha,\bm\beta)=\bm g(0, \theta, \bm\alpha,\bm\beta)=\bm y_0$ uniformly for $\theta\in [0, 1]$, $\bm\alpha\in [0, 1]^{n-\ell}$ and $\bm\beta\in [0,1]^
\ell$. By this and the uniqueness of the positive solution of \eqref{n patch system}, we obtain the desired result.  
\end{proof}

\begin{proposition}\label{prop:u_alpha}
We have
\begin{align*}
&\lim_{r\to\infty} u_i' = \begin{cases}
\frac{-\alpha_iH}{c_i} \quad &\text{for}\quad  i=\ell+1,\dots,n\\
\frac{\beta_i H}{c_i} \quad &\text{for}\quad  i=1,\dots,\ell.
\end{cases}    
\end{align*}
Furthermore the convergence is uniform in $\bm\alpha\in [0, 1]^{n-\ell}$, $\bm\beta\in[0,1]^{\ell}$ and $\theta\in [0, 1]$. 
\end{proposition}

\begin{proof}
By \eqref{eq:u_alpha_2}-\eqref{eq:u_beta_2}, $\bm u'=(u_1', \dots, u_n')^T$ satisfies $B\bm u'=\bm R$ with $R=(R_1, \dots, R_n)^T$, $R_i=\alpha_i H$ for $i=\ell+1,\dots,n$ and $R_{i}=-\beta_i H$ for $i=1,\dots,\ell$, and 
\begin{equation*}
B_{ij}=
\left\{
\begin{array}{lll}
\frac{a_{ij}}{u_i}, \ \ &{i\neq j},\\
-\sum_{j} a_{ij}\frac{u_j}{u_i^2}-c_i,  &i=j. 
\end{array}
\right.
\end{equation*}
By Lemma \ref{lemma:u_alpha},  we know $B\to -\text{diag}(c_1,\dots, c_n)$ as $r\to\infty$ uniformly in $\bm\alpha\in [0, 1]^{n-\ell}$, $\bm\beta\in[0,1]^{\ell}$ and $\theta\in [0, 1]$. The desired result follows from this observation. 
\end{proof}

Using the assumption that $c_1=\dots=c_\ell=c>c_i$ for all $i=\ell+1,\dots,n$, and the fact that $\sum_{i=1}^\ell\beta_i=1$ and $\sum_{i=\ell+1}^n\alpha_i=1$, we have the following corollary.

\begin{corollary}\label{cor:u_alpha}
We have
\[
\lim_{r\to\infty}\frac{1}{H}\MM'= \frac{1}{c}-\sum_{i=\ell+1}^n\frac{\alpha_i}{c_i} < \frac{1}{c}-\frac{1}{\max_{j\in[\ell+1,n]}c_j} <0.
\]    
\end{corollary}

By Corollary \ref{cor:u_alpha}, the total biomass $\MM$ is decreasing in $\theta$ when $r$ is large. This implies the best strategy to maximize $\MM$ is to concentrate the harvesting effort on patches $1,\dots,\ell$, i.e. the patches with the highest intraspecies competition rate.

\subsubsection{Harvesting effort distribution among the patches with the highest intraspecies competition rate}

The remaining question is to determine which harvesting effort distribution among the patches with the highest intraspecies competition rate yields the maximum biomass. We show that the answer depends on the structure of the movement network. In particular, we define a quantity for each node as follows.
\begin{definition}\label{def:netflow}
    For each node $i$, the \textit{effective net flow} $I_i$ is defined by
    \[
    I_i= \sum_j\bigg(\frac{a_{ij}}{c_j}-\frac{a_{ji}}{c_i} \bigg).
    \]
\end{definition}
We show that to maximize the total biomass, among the patches $i=1,\dots,\ell$, we concentrate all harvesting effort on the patch with the highest effective net flow. Without loss of generality, assume $I_1>I_i$ for all $i=2,\dots,\ell$.  To this end, let
\begin{align}\label{hi_3third}
h_1=(1-\delta)H, \quad  h_i=\gamma_i\delta H \quad \text{for} \  i=2,\dots,\ell, \quad \text{and}\quad h_i=0 \quad \text{for}\ i=\ell+1, \dots, n. 
\end{align}
where $\sum_{i=2}^\ell \gamma_i=1$ and $\delta\in[0,1]$.

The following results, Lemma \ref{lemma:u_gamma} and Proposition \ref{prop:u_gamma}, are similar to Lemma \ref{lemma:u_alpha} and Proposition \ref{prop:u_alpha} and thus their proofs are omitted.

\begin{lemma}\label{lemma:u_gamma}
We have
    \[
    \lim_{r\to\infty} \frac{u_i}{r} = \frac{1}{c_i}.
    \]
    Furthermore the convergence is uniform in $\bm\gamma\in [0, 1]^{\ell-1}$ and $\delta\in [0, 1]$.
\end{lemma}
\begin{proposition}\label{prop:u_gamma}
We have
\begin{align*}
&\lim_{r\to\infty} u_i' = \begin{cases}
\frac{H}{c_i} =\frac{H}{c} \quad &\text{for}\quad  i=1\\
\frac{-\gamma_i H}{c_i}=\frac{-\gamma_i H}{c} \quad &\text{for}\quad  i=2,\dots,\ell\\
0 \quad &\text{for}\quad  i=\ell+1,\dots,n.
\end{cases}    
\end{align*}
 Furthermore the convergence is uniform in $\bm\gamma\in [0, 1]^{\ell-1}$ and $\delta\in [0, 1]$.
\end{proposition}

Based on Proposition \ref{prop:u_gamma} and the assumption that $c_1=\dots=c_\ell$, we have $\lim_{r\to\infty}\MM' = 0$. To determine the sign of $\MM'$ for large $r$, we compute the limit of $r\MM'$ instead. 

\begin{proposition}\label{prop:rM'}
      We have
    \begin{align*}
        \lim_{r\to\infty} r\MM' = H(-I_{1}+\sum_{i=2}^\ell\gamma_i I_i).
    \end{align*}
    Furthermore the convergence is uniform in $\bm\gamma\in [0, 1]^{\ell-1}$ and $\delta\in [0, 1]$. As a consequence, we have 
     \[
    \lim_{r\to\infty} \frac{r}{H}\MM'=-I_{1}+\sum_{i=2}^\ell\gamma_i I_i < - I_1+\max_{i\in[2,\ell]}I_i<0. 
    \]   
\end{proposition}

\begin{proof}
Using the same calculations that lead to \eqref{eq:u_alpha_2} and \eqref{eq:u_beta_2}, we can obtain the following equations
\begin{align*}
    &-u_i^2(-H+c_iu_i')+\sum_j a_{ij}(u_j'u_i-u_ju_i')=0 \quad \text{for} \quad i=1\\
    &-u_i^2(\gamma_iH+c_iu_i')+\sum_j a_{ij}(u_j'u_i-u_ju_i')=0 \quad \text{for} \quad i=2,\dots,\ell\\
    &-u_i^2c_iu_i'+\sum_j a_{ij}(u_j'u_i-u_ju_i')=0 \quad \text{for} \quad i=\ell+1,\dots,n.
\end{align*}
Multiplying each equation by $\frac{r}{c_iu_i^2}$ and rearranging terms yields
\begin{align*}
    &ru_i' - \frac{Hr}{c_i}=\frac{r}{c_iu_i} \sum_ja_{ij}(u_j'-\frac{u_j}{u_i}u_i')\quad \text{for} \quad i=1\\
    &ru_i' + \frac{\gamma_iHr}{c_i}=\frac{r}{c_iu_i} \sum_ja_{ij}(u_j'-\frac{u_j}{u_i}u_i') \quad \text{for} \quad i=2,\dots,\ell\\
    &ru_i'=\frac{r}{c_iu_i} \sum_ja_{ij}(u_j'-\frac{u_j}{u_i}u_i') \quad \text{for} \quad i=\ell+1,\dots,n.
\end{align*}
Adding all equations above and using the assumption that $c_1=\dots=c_\ell=c$ and $\sum_{i=2}^\ell \gamma_i=1$ yields
\begin{align*}
    r\MM'= \sum_i \frac{r}{c_iu_i} u_i'\sum_j \left(a_{ji} - a_{ij} \frac{u_j}{u_i}\right).
\end{align*}
Finally, taking the limit $r\to\infty$ and using Lemma \ref{lemma:u_gamma} and Proposition \ref{prop:u_gamma} gives us the desired result. 
\end{proof}

From Proposition \ref{prop:rM'}, we conclude that when $r$ is large, the total biomass is maximized when $\delta=0$, i.e. among the patches with the highest $c_i$, we concentrate all harvesting effort on the patch with the highest effective net flow.

\begin{figure}[htbp]
\centering
\begin{tikzpicture}
\begin{scope}[every node/.style={draw}, node distance= 1.5 cm]
           \node[circle] (7) at (4.5,0) {$1$};
           \node[circle] (8) at (13.5,0) {$3$};
    \node[circle] (9) at (9,-3) {$4$};
    \node[circle] (10) at (9, 0) {$2$};
      \node[circle] (11) at (9, -6) {$5$};
\end{scope}
\begin{scope}[every node/.style={fill=white},
              every edge/.style={thick}]
        
       \draw[thick] [->](7) to [bend right] node[left=5] {{\footnotesize $d+q$}} (9);
        \draw[thick] [->](9) to node[right=5] {{\footnotesize $d$}} (7);  
            \draw[thick] [->](8) to [bend left] node[right=5] {{\footnotesize $d+q$}} (9);
        \draw[thick] [->](9) to node[left=5] {{\footnotesize $d$}} (8);
        \draw[thick] [->] (10) to [bend right] node[left=3] {{\footnotesize $d+q$}} (9);
         \draw[thick] [->] (9) to [bend right] node[right=3] {{\footnotesize $d$}} (10);
        \draw[thick] [->] (9) to [bend right] node[left=3] {{\footnotesize $d+q$}} (11);
         \draw[thick] [->] (11) to [bend right] node[right=3] {{\footnotesize $d$}} (9);
\end{scope}
\end{tikzpicture}
\caption{A stream network with structure 3–1–1 (three patches on top, one in the middle, one
below).}\label{river}
\end{figure}
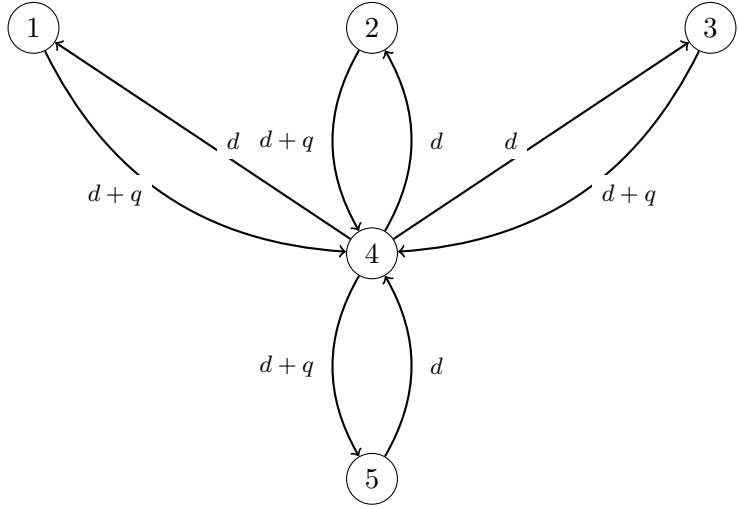

\begin{example}
Consider an $n$-patch straight stream network with homogeneous patches (i.e., $r_i=r$ and $c_i=c$ for all $i=1,\dots,n$). Assume further that the drift and diffusion are also homogeneous, i.e., the movement rate from an upstream patch to a downstream patch is $d+q$, and the movement rate from a downstream patch to an upstream patch is $d$. Then it is easy to check that the most downstream patch has the maximum effective net flow. Thus, when $r$ is large, we can maximize the (remaining) total biomass by harvesting the most downstream patch exclusively.

\end{example}

\begin{example}
The most downstream patch is not always the patch with the maximum effective net flow even if we assume homogeneous patches and homogeneous drift and diffusion. In this example, we consider the 5-patch stream network shown in Figure \ref{river}. It is straightforward to verify that patch 4 has the maximum effective net flow.
We illustrate this result numerically. In Figure~\ref{fig:311biomass}, we plot $\Delta \mathcal{M}^\ast = \mathcal{M}^\ast_{4} - \mathcal{M}^\ast_{\text{strategy}}$, where $\mathcal{M}^\ast_{4}$ is the resulting biomass when only patch 4 is harvested. Thus, values above $0$ indicate that harvesting only the middle patch (patch 4) yields more biomass than the alternative strategy, while values below $0$ indicate the alternative leaves more biomass. The solid green curve corresponds to harvesting evenly across the three top (periphery) patches $[\tfrac{1}{3},\tfrac{1}{3},\tfrac{1}{3},0,0]$; the orange dashed curve is harvesting only the upstream-left patch $[1,0,0,0,0]$; the purple dotted curve is harvesting only the bottom patch $[0,0,0,0,1]$; the red dash--dotted curve is uniform harvesting over all five patches $[\tfrac{1}{5},\tfrac{1}{5},\tfrac{1}{5},\tfrac{1}{5},\tfrac{1}{5}]$; and the black dash--dotted line at $0$ is the reference. For small $r$, the patch 1-only and patch $5$-only strategies (orange, purple) yield slightly \emph{more} biomass than  patch $4$ (negative $\Delta \mathcal{M}^\ast$), but as $r$ increases they cross the reference near $r\approx25$ and thereafter leave less biomass than harvesting only patch 4.

 \begin{figure}
    \centering
    \includegraphics[width=0.8\linewidth]{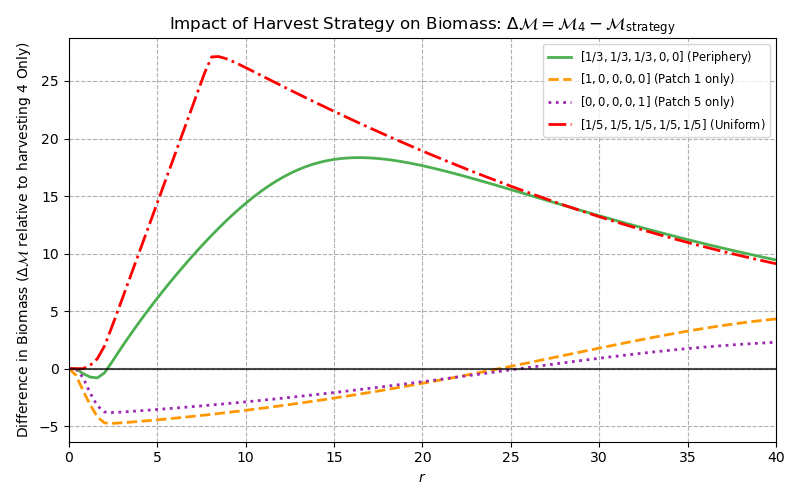}
    \caption{For a stream network with structure 3--1--1 (three patches on top, one in the middle, one below), the figure compares stabilized total biomass across five harvesting strategies as the intrinsic growth rate $r$ increases. }

    \label{fig:311biomass}
\end{figure}

\end{example}

\subsection{Maximizing the yield}
\subsubsection{Concentrating on the patches with the smallest intraspecies competition rate}

Without loss of generality, we may assume that $c_1=\dots=c_{\ell}=c<c_i$ for $i=\ell+1,\dots,n$, i.e. the first $\ell$ patches have the smallest intraspecies competition rate.

We first show that to maximize the yield, it is better to concentrate the harvesting effort in patch $1,\dots,\ell$. To see this, let the harvesting effort in each patch be given as follows:
\begin{align}\label{hiY}
&h_i=\begin{cases}
    \beta_i(1-\theta)H \quad &\text{for} \quad i=1,\dots,\ell\\
    \alpha_i\theta H \quad &\text{for} \quad i=\ell+1,\dots,n
\end{cases}
\end{align}
where $\sum_{i=1}^\ell \beta_i=1$, $\sum_{i=\ell+1}^n \alpha_i=1$, and $\theta\in[0,1]$. We now show that for any choice of $\bm\alpha \in \mathbb{R}^{n-\ell}$ and $\bm\beta\in\mathbb{R}^\ell$, we must have $\lim_{r\to\infty}\YY'<0$, where the differentiation is with respect to $\theta$. Using Lemma \ref{lemma:u_alpha} and Proposition \ref{prop:u_alpha} we obtain the following corollary.

\begin{corollary}\label{cor:YY'}
    We have 
    \[
    \lim_{r\to\infty}\frac{\YY'}{rH} =-\frac{1}{c}+\sum_{i=\ell+1}^n \frac{\alpha_i}{c_i} <-\frac{1}{c}+\frac{1}{\min_{j\in[\ell+1,n]}c_j} <0.
    \]
Furthermore the convergence is uniform in $\bm\alpha\in [0, 1]^{n-\ell}$, $\bm\beta\in[0,1]^{\ell}$, and $\theta\in [0, 1]$.
\end{corollary}
\begin{proof}
We have
    \[
    \YY=\sum_{i=1}^\ell \beta_i(1-\theta)Hu_i+\sum_{i=\ell+1}^n \alpha_i\theta H u_i.
    \]
Thus
\[
\frac{\YY'}{rH}=\frac{1}{r}\sum_{i=1}^\ell (u_i'\beta_i(1-\theta)-u_i\beta_i)+\sum_{i=\ell+1}^n (u_i'\alpha_i \theta + u_i\alpha_i).
\]
Using Lemma \ref{lemma:u_alpha}, Proposition \ref{prop:u_alpha}, and the assumption that $c_1=\dots=c_\ell=c<c_i$ for $i=\ell+1,\dots,n$, we have 
\[
\lim_{r\to\infty} \frac{\YY'}{rH} = -\sum_{i=1}^\ell \frac{\beta_i}{c_i}+\sum_{i=\ell+1}^n\frac{\alpha_i}{c_i} =-\frac{1}{c}+\sum_{i=\ell+1}^n \frac{\alpha_i}{c_i} <-\frac{1}{c}+\frac{1}{\min_{j\in[\ell+1,n]}c_j} <0.
\]
\end{proof}

By Corollary \ref{cor:YY'}, the yield $\YY$ is decreasing in $\theta$ when $r$ is large. This implies the best strategy to maximize $\YY$ is to concentrate the harvesting effort on patches $1,\dots,\ell$, i.e. the patches with the smallest intraspecies competition rate.

\subsubsection{Harvesting effort distribution among the patches with the smallest intraspecies competition rate}

We now determine which harvesting effort distribution among the patches with the smallest intraspecies competition rate yields the maximum yield. We show that the answer depends on the structure of the movement network. 
 In particular, we show that, under some technical conditions, we can maximize the yield by concentrating the harvesting effort on the patch with the highest effective net flow. Without loss of generality, assume $I_1>I_i$ for all $i=2,\dots,\ell$. To this end, let
\begin{align}\label{hi_3again}
h_1=(1-\delta)H, \quad  h_i=\gamma_i\delta H \quad \text{for} \  i=2,\dots,\ell, \quad \text{and} \quad h_i=0\quad \text{for}\ i=\ell+1, \dots, n,
\end{align}
where $\sum_{i=2}^\ell \gamma_i=1$ and $\delta\in[0,1]$.

It is easy to verify that Lemma \ref{lemma:u_gamma} and Proposition \ref{prop:u_gamma} still hold.

\begin{proposition}\label{prop:Y'}
    We have
    \[
    \YY'=\frac{2H}{c}(1-\delta-\delta \sum_{i=2}^\ell\gamma_i^2) -I_1+\sum_{i=2}^\ell \gamma_iI_i.
    \]
\end{proposition}
\begin{proof}
 We have  
     \[
    \YY=(1-\delta)Hu_1 +\sum_{i=2}^\ell \gamma_i\delta Hu_i.
    \]
Thus
\begin{align}\label{eq:Y'H}
\frac{\YY'}{H}=(1-\delta)u_1'+\delta\sum_{i=2}^\ell \gamma_iu_i'-u_1+\sum_{i=2}^\ell \gamma_iu_i = (1-\delta)u_1'+\delta\sum_{i=2}^\ell \gamma_iu_i' + \sum_{i=2}^\ell \gamma_i(u_i-u_1).
\end{align}
By dividing each equation at equilibrium by the corresponding $u_i$, we obtain
\begin{align}
    &r-(1-\delta)H-cu_1+\frac{1}{u_1}\sum_j (a_{1j}u_j-a_{j1}u_1)=0\label{eq:u_1_yy}\\
    &r-\gamma_i\delta H-cu_i+\frac{1}{u_i}\sum_j (a_{ij}u_j-a_{ji}u_i)=0, \ i=2, \dots, \ell.\label{eq:u_i_yy}
\end{align}
Taking the difference \eqref{eq:u_1_yy} $-$ \eqref{eq:u_i_yy} yields
\begin{align}
    c(u_i-u_1)=(1-\delta-\gamma_i\delta)H+\frac{1}{u_i}\sum_j (a_{ij}u_j-a_{ji}u_i) - \frac{1}{u_1}\sum_j (a_{1j}u_j-a_{j1}u_1),
\end{align}
for $i=2,\dots, \ell$.
Taking the limit $r\to\infty$ and using Lemma \ref{lemma:u_gamma}  yields
\begin{align}\label{eq:ui-u1}
\lim_{r\to\infty} (u_i-u_1)=\frac{(1-\delta-\gamma_i\delta)H}{c}+I_i-I_1.
\end{align}
Combining equations \eqref{eq:Y'H}, \eqref{eq:ui-u1}, and the results in Proposition \ref{prop:u_gamma}, we have 
\[
\lim_{r\to\infty} \frac{\YY'}{H}= \frac{2H}{c}(1-\delta-\delta \sum_{i=2}^\ell\gamma_i^2) -I_1+\sum_{i=2}^\ell \gamma_iI_i.
\]
\end{proof}
\begin{remark}\normalfont
Proposition \ref{prop:Y'} allows us to obtain generalizations of the results in Section \ref{sec:yield}. Firstly, if the movement rates are large enough compare to the total harvesting effort $H$, then we can maximize the yield by concentrating the harvesting effort on the patch with the maximum effective net flow (among the patches with the smallest instraspecies competition rates). 
\end{remark} 

The corollary below follows directly from Proposition \ref{prop:Y'}.

\begin{corollary}\label{cor:maxY}
    Suppose $I_1-\max_{i\in[2,\ell]} I_i > \frac{2H}{c}$, then
    \[
    \lim_{r\to\infty} \frac{\YY'}{H} < \frac{2H}{c} - (I_1-\max_{i\in[2,\ell]} I_i ) <0.
    \]
    Thus $\YY(\delta)$ is maximized when $\delta=0$, i.e. all harvesting effort is concentrated in patch 1, which is the patch with the highest effective net flow.
\end{corollary}

\begin{example}
    For $n$-patch straight stream network with homogeneous patches and homogeneous flow, the effective net flow of patch 1 is $-q/c$, the effective net flow of patch $n$ is $q/c$, and the rest of the patches have 0 effective net flow. So the condition in Corollary \ref{cor:maxY} becomes $q>2H$, which is similar to the condition in Theorem \ref{thm:max_Y_downstream}.
\end{example}

Even when the condition in Corollary \ref{cor:maxY} is not fulfilled, we can still obtain some partial results similar to Corollary \ref{cor:max_Y}. It is easy to see that  Corollary \ref{cor:max_Y} is a special case of Corollary \ref{cor:H/n} when $n=2$.

\begin{corollary}\label{cor:H/n}
Suppose that $\delta>\frac{n-1}{n}$, then  $\lim_{r\to\infty} \frac{\YY'}{H}<0$. Thus the yield $\YY$ is maximized at some $\delta^*<\frac{n-1}{n}$. In other words, to maximize the yield the harvesting effort in patch 1 (the patch with the highest effective net flow) must exceed $H/n$.   
\end{corollary}
\begin{proof}
    We have the inequality
    \[
    \sum_{i=2}^\ell \gamma_i^2 \geq \frac{(\sum_{i=2}^\ell\gamma_i)^2}{n-1} = \frac{1}{n-1}.
 \]
 Thus from Proposition \ref{prop:Y'}, if $\delta>(n-1)/n$, we have
 \begin{eqnarray*}
   \lim_{r\to\infty} \frac{\YY'}{H} &=& \frac{2H}{c}(1-\delta-\delta
 \sum_{i=2}^\ell\gamma_i^2) -I_1+\sum_{i=2}^\ell \gamma_iI_i \\
 &<& \frac{2H}{c}(1-\delta\frac{n}{n-1})-I_1+\sum_{i=2}^\ell \gamma_iI_i <\frac{2H}{c}(1-\delta\frac{n}{n-1})<0.   
 \end{eqnarray*}
\end{proof}

\section{Discussion}\label{section discussion}

We studied how to allocate a fixed harvesting budget across patches in a stream network to optimize two objectives: maximizing the total biomass and maximizing sustainable yield in the presence of biased movement. Overall, the findings suggest that for a stream network with spatial heterogeneity, the patch parameters determine the optimal strategy whereas for homogeneous stream network, the network structure determines the optimal strategies.  

For the homogeneous two-patch model, we obtained an analytical classification for maximizing biomass. Our findings show that when resources are abundant, harvesting exclusively on the downstream patch gives the largest total biomass, this result agrees with~\cite{nguyen2023maximizing} where concentrating resources on the upstream patch yields the largest biomass. On the other hand, when resources are modest, harvesting only upstream results in maximum biomass. Given that~\cite{nguyen2023maximizing} shows concentrating resources downstream promotes the growth rate (that is favoring downstream, similar to harvest exclusively upstream here), this could indicate the importance of promoting the growth rate when the resources are modest or scarce.

In addition, our results indicate a similar strategy for maximizing yield for the two-patch stream network, but with much more complexity. We derived parameter conditions guaranteeing that harvesting exclusively downstream maximizes yield, particularly in a high-growth regime with sufficiently strong advection relative to the harvest budget. Outside this regime, the yield-maximizing strategy can occur upstream, downstream, or at an interior split, consistent with the numerical examples and highlighting the stronger sensitivity of yield to the spatial distribution of equilibrium densities.

For $n$-patch networks, our analysis focuses on the case when $r$ is large, which allows a tractable separation between (a) which subset of patches should receive harvest and (b) how to allocate harvest within that subset. With heterogeneous intraspecific competition rates, we found the opposite strategy depends on the objective: maximizing  biomass favors concentrating harvest on patches with larger competition rates, while maximizing yield favors concentrating harvest on patches with smaller competition rates. However, for homogeneous stream network, the structure enters through an effective net-flow metric that identifies where concentrating harvest is most effective; in particular, the decisive patch need not be the most downstream node when connectivity creates strong transport asymmetries and the same strategy apply to maximize both biomass and yield.

While this paper focuses on two optimization problems of maximizing total biomass and yield under harvesting constraints, other optimization questions naturally arise in the context of heterogeneous population control or preservation, as well as in infectious disease management (e.g., see~\cite{d2012efficacy}) and tumor/cancer control (e.g., see~\cite{hillen2009user}). We plan to investigate some of these questions in future work:
\begin{itemize}
\item
{\bf Minimum total control effort for extinction.} Find the least total control required across all patches to drive the metapopulation to extinction. That is, 
$$\min_{h_i\ge 0} \; \sum_{i=1}^n h_i$$
subject to $u_i(t) \to 0$ as $t\to \infty$, $i=1,2,\ldots, n$.

\item {\bf Optimal patch selection.} Minimize the number of patches under active control while still achieving extinction: 
$$\min \#\{i: h>0\}$$
subject to $u_i(t)\to 0$ as $t\to\infty$, $i=1,2,\ldots, n$.

\item {\bf Targeted control under budget constraints.} Distribute a fixed control budget $H>0$ across patches to minimize the long-term total population:
$$\min_{h_i \ge 0} \; \limsup_{t\to\infty} \sum_{i=1}^n u_i(t)$$
subject to $\sum_{i=1}^n h_i \le H$.
\end{itemize}

\section*{Acknowledgments}
The authors thank the American Institute of Mathematics
(AIM) for hosting and generously supporting an AIM SQuaRE program focusing on
population persistence in stream networks, at which this research was initiated and developed. This research was partially supported by the National Science Foundation (NSF) through the grants DMS 2527228 (ZS), DMS 2532769 (YW), and by the Simons Foundation (YW; ZS).

\bibliographystyle{plain}
\bibliography{ref}

\end{document}